\documentclass[10pt]{article}
\usepackage{amsmath, amsthm, amssymb}
\usepackage{graphicx,authblk}
\usepackage[all]{xy}
\xyoption{all}
\newtheorem{corollary}{Corollary}
\newtheorem{proposition}{Proposition}

\newtheorem{theorem}{Theorem}
\newtheorem*{theorem*}{Theorem}
\newtheorem*{thmA}{Theorem A}
\newtheorem{definition}{Definition}

\newtheorem{question}{Question}
\newcommand\wt{\widetilde}
\newcommand\ov{\overline}

\newcommand\surj{\twoheadrightarrow}

\begin{document}
\title{Fundamental Theorems for the $K$-theory of {\bf S}-algebras, I:\\ the connective case}
\author{C. Ogle}
\maketitle

\section{Introduction} The Fundamental Theorem of $K$-theory (first formulated by Bass in low dimensions and later extended by Quillen to all dimensions \cite{dg}) yields an isomorphism
\[
K_*(R[t,t^{-1}])\cong K_*(R)\oplus K_{*-1}(R)\oplus NK^+_*(R)\oplus NK^-_*(R)
\]
where $R$ is a discrete ring, $K_*(_-)$ denotes its Quillen $K$-groups, and
\[
NK^{\pm}_*(R)\cong NK_*(R) := \ker(K_*(R[t])\overset{t\mapsto 0}{\longrightarrow} K_*(R))
\]
The groups here are possibly non-zero in negative degrees, given that they are computed as the homotopy groups of a (potentially) non-connective delooping of the Quillen $K$-theory space, arising from a spectral formulation of this result \cite{cw2}. The nil-groups $NK_*(R)$ capture subtle \lq\lq tangential\rq\rq\ information about $R$, and are remarkably difficult to compute. It is the purpose of this series of papers to extend this fundamental theorem to the Waldhausen $K$-theory of {\bf S}-algebras, and investigate the structure of the corresponding nil-groups.
\vskip.2in
As a first step we consider the connective case. Specifically, let ${\cal{CSA}}$ denote the category of connective ($=<-1>$-connected) {\bf S}-algebras and {\bf S}-algebra homomorphisms, in the sense of \cite{ekmm}. Write $K(_-)$ resp.\ $TC(_-):{\cal{CSA}}\to (spectra)_*$ for the functor which associates to an {\bf S}-algebra its (connective) Waldhausen $K$-theory resp.\ topological cyclic homology spectrum. In \cite{dgm} the authors construct a natural transformation, referred to as the {\it topological trace map} (modeled on the original Dennis trace)
\begin{equation}
Tr(_-):K(_-)\to TC(_-)
\end{equation}
and prove the following:

\begin{thmA} [DGM] Let $f:A\to B$ be a morphism in $\cal{CSA}$ which is surjective on $\pi_0$ with nilpotent kernel. Then $Tr$ induces an equivalence of spectra
\[
hofib(K(f):K(A)\to K(B))\overset{\simeq}{\to} hofib(TC(f):TC(A)\to TC(B))
\]
\end{thmA}

We show how the techniques used to prove this result, in conjunction with Goodwillie's Calculus of Functors \cite{g1, g2, g3}, allow for a geodesic extension of the Bass-Quillen theorem to the (Waldhausen) $K$-theory of connective {\bf S}-algebras (in the process recovering the main result of \cite{hkvww1}). Moreover, the same methods can be used to prove two localization theorems for the relative 
$K$-theory associated to a 1-connected morphism of connective {\bf S}-algebras, something we show first. 
\vskip.2in

We would like to thank J. Klein, A. Salch, P. Goerss, J. Rognes, J. Harper, and G. Mislin for helpful communications and conversations during the preparation of this paper.
\vskip.3in


\section{Relative $K$-theory and localization}


\subsection{Connective localization} Connective localization is of a very particular type; nevertheless this first result is of interest due to its method of proof. Following \cite{bk1}, we call an {\bf S}-algebra $B$ {\it solid} if $B\wedge B = B\underset{S}{\wedge} B \simeq B$, with the equivalence induced by the multiplication on $B$. If $f:A\to B$ is a homomorphism of {\bf S}-algebras, set $K(f:A\to B) := hofib(K(f):K(A)\to K(B))$.  

\begin{theorem}\label{thm:main} Let $f:A\to B$ be a 1-connected morphism in $\cal{CSA}$, and $R$ a solid, unital, connective {\bf S}-algebra. Then there is a natural weak equivalence
\[
K(f:A\to B)\wedge R\simeq K(f\wedge id: A\wedge R\to B\wedge R)
\]
\end{theorem}

\begin{proof} Fixing $E$, define functors $F_i, 1\le i\le 3$ on ({\bf S}-algebras) by
\begin{gather*}
F_1(A) := K(A)\wedge R,\\
F_2(A) := K(A\wedge R),\\
F_3(A) := K(A\wedge R)\wedge R
\end{gather*}
The unit map $S\to R$ determines natural transformations $F_1(_-)\to F_3(_-)\leftarrow F_2(_-)$. We follow the program of \cite{dgm} to first provide a sequence of reductions.
\vskip.2in

\underbar{Claim 1}\ \ The functors $F_i$ are determined by their value on simplicial rings. In fact, the functor $A_{\bullet}\mapsto A_{\bullet}\wedge R$ (degreewise smashing with $R$) preserves both simplicial weak equivalences and connectivity in the simplicial coordinate. The method of \cite[III.3]{dgm} applies to yield the result (as in {\it loc.\ cit.}, one allows {\bf S}-algebra maps between simplicial rings).
\vskip.2in

Now suppose given a 1-connected map $f_{\bullet}:A_{\bullet}\to B_{\bullet}$ of simplicial rings. Resolving by tensor algebras over $S$, we may assume $f_{\bullet}$ is represented by a split surjection with $I_{\bullet} := ker(f_{\bullet})$ a 0-reduced simplicial ideal. Filtering by powers of $I_{\bullet}$, we conclude
\vskip.2in

\underbar{Claim 2}\ \  If the statement of the theorem is true for split square-zero extensions of simplicial rings with 0-reduced kernel, then it is true in general.
\vskip.2in

Representing the split square-zero extension as $A_{\bullet}\ltimes M_{\bullet}\surj A_{\bullet}$ ($M_{\bullet}$ an $A_{\bullet}$ bi-module), we can resolve $M_{\bullet}$ by a free bimodule $A_{\bullet}[Y_{\bullet}]\overset{\simeq}{\surj} M_{\bullet}$, yielding the further reduction to
\vskip.2in

\underbar{Claim 3}\ \  If the statement of Theorem \ref{thm:main} is true for all split square-sero extensions of the form $A_{\bullet}\ltimes A_{\bullet}[Y_{\bullet}]\surj A_{\bullet}$ for 0-reduced $Y_{\bullet}$, then it is true in general.
\vskip.2in

These type of extensions are studied by Lindenstrauss and McCarthy in \cite{lm}, where the homotopy fibre $hofib(K(A_{\bullet}\ltimes A_{\bullet}[Y_{\bullet}])\surj K(A_{\bullet}))$ is abbreviated as $\wt{K}(A_{\bullet};\wt{A}_{\bullet}[Y_{\bullet}])$ (the \lq\lq tilde\rq\rq\ used so as to distinguish these groups from the stable $K$-groups of $A_{\bullet}$ with coefficients in $A_{\bullet}[Y_{\bullet}]$). This functor can be represented as a homotopy functor on basepointed simplicial sets (aka spaces)
\[
X_{\bullet}\mapsto \wt{K}(A_{\bullet};\wt{A}_{\bullet}[X_{\bullet}])
\]
We define three functors on basepointed simplicial sets:
\begin{gather*}
\wt{F}_1(X_{\bullet}) := \wt{K}(A_{\bullet}\wedge R; \wt{A}_{\bullet}[X_{\bullet}]\wedge R);\\
\wt{F}_2(X_{\bullet}) := \wt{K}(A_{\bullet}; \wt{A}_{\bullet}[X_{\bullet}])\wedge R;\\
\wt{F}_3(X_{\bullet}) := \wt{K}(A_{\bullet}\wedge R; \wt{A}_{\bullet}[X_{\bullet}]\wedge R)\wedge R.
\end{gather*}

Again, there are natural transformations of homotopy functors $\wt{F}_1(_-)\to \wt{F}_3(_-)\leftarrow \wt{F}_2(_-)$. Although not originally stated in this generality, the methods of \cite{lm} apply essentially without change to show

\begin{theorem*} [LM] The homotopy functors $\wt{F}_i(_-)$ are 0-analytic.
\end{theorem*}

It remains to compare coefficients of Goodwillie derivatives. This is accomplished by

\begin{theorem*} [LM] For each $n\ge 1$, the n-th Goodwillie derivatives of $\wt{F}_i(_-), 1\le i\le 3$ are represented (respectively) as spectra with $C_n$-action by
\begin{gather*}
sd_n THH(A_{\bullet};A_{\bullet})\wedge R;\\
sd_n THH(A_{\bullet}\wedge R;A_{\bullet}\wedge R);\\
sd_n THH(A_{\bullet}\wedge R;A_{\bullet}\wedge R)\wedge R.
\end{gather*}
where $sd_n$ is the n-th edgewise subdivision functor.
\end{theorem*}

But now for $THH$, as well as its edgewise subdivisions, there are natural equivalences
\[
sd_n THH(A_{\bullet};A_{\bullet})\wedge R\overset{\simeq}{\to} sd_n THH(A_{\bullet}\wedge R;A_{\bullet}\wedge R)\wedge R
\overset{\simeq}{\leftarrow} sd_n THH(A_{\bullet}\wedge R;A_{\bullet}\wedge R)
\]
arising from the equivalence $R\simeq R\wedge R$, together with the unit map $S\to R$. By convergence of the respective Goodwillie towers, we conclude there are equivalences of homotopy functors on the category of basepointed connected simplicial sets
\[
\wt{F}_1(_-)\overset{\simeq}{\to} \wt{F}_3(_-)\overset{\simeq}{\leftarrow} \wt{F}_2(_-)
\]
which completes the proof of Theorem \ref{thm:main}.
\end{proof}

By Theorem A, we have

\begin{corollary} Let $f:A\to B$ be a 1-connected homomorphism of connective {\bf S}-algebras, with $R$ connective, unital, and solid. Then there are equivalences of spectra
\[
TC(f:A\to B)\wedge R\simeq TC(f\wedge id:A\wedge R\to B\wedge R)
\]
where $TC(_-):({\bf S}$-$algebras)\to (spectra)$ associates to an {\bf S}-algebra its topological cyclic homology spectrum.
\end{corollary}

\underbar{\bf Remark 1}  It is worth comparing the statement of Theorem \ref{thm:main} with the original results of Weibel concerning localization (at a prime $p\in\mathbb Z$) of the relative algebraic $K$-groups associated to a nilpotent ideal \cite{cw1}.
\vskip.2in

The conditions imposed on $R$ above are, unfortunately, quite restrictive. An equivalent condition was considered by Rognes in \cite{jr}. Precisely, the hypothesis on $R$ is equivalent to the unit map $S\to R$ being smashing in the sense of \cite{jr}, where Rognes shows that this implies $R$ is the smashing localization of the sphere spectrum with respect to a homology theory $E$; the connectivity of $E$ then implies it occurs as the ordinary arithmetic localization of the sphere spectrum at some set of primes. In other words,

\begin{proposition} If $R$ is a solid, unital, connective {\bf S}-algebra, then $R = M\left(R'\right)$, the Moore spectrum of a discrete ring $R'\subset\mathbb Q$.
\end{proposition}

The above theorem recovers, in the framework of (connective) {\bf S}-algebras, a localization result first proven by Waldhausen in \cite{fw1}.
This proposition is consistent with the fact the Bousfield class $<L>$ of any connective localization functor must be of the form $<MSG>$, the class associated to the Moore spectrum of an abelian group $G$ \cite{akb,dr}.
\vskip.2in


\subsection{Non-connective localization and the chromatic tower} As we have noted, solid {\bf S}-algebras are closely connected with localization functors. Recall \cite{akb,dr} that a localization functor $L$ is {\it smashing} if there exists a (non)-connective spectrum $E$ with $L(X)\simeq X\wedge E$ for any spectrum $X$. In this case we will sometimes write $L$ as $L_E$, with $L(X)$ referred to as the localization of $X$ at $E$. For such a smashing localization functor, the spectrum $E$ is recovered up to homotopy as $E = L_E(S)$, where $S$ denotes the sphere spectrum.  
\vskip.1in
A map of spectra $X\to Y$ is an {\it $E$-equivalence} if it becomes a weak equivalence upon localizing at $E$; a spectrum is {\it $E$-acyclic} if its localization at $E$ is contractible.
\vskip.2in

\begin{theorem}\label{thm:second} Let $f:A\to B$ be a 1-connected morphism in $\cal{CSA}$, and $L_E$ a smashing localization functor. If $f$ is an $E$-equivalence, then so is $K(f):K(A)\to K(B)$.
\end{theorem}

\begin{proof} By Theorem A, there is an equivalence
\[
K(f:A\to B)\simeq TC(f:A\to B)
\]
For a set $X$, let $T_S(X)$ denote the tensor algebra over $S$ generated by $X$. Given a simplicial set $X_\bullet$, define functors $(simplicial\ sets)_*\to (spectra)_*$ by
\begin{gather*}
F_{R}(X_\bullet) := \left|[q]\mapsto  TC(R\wedge T_S(X_q))\right|,\quad R = A,B\\
G(X_\bullet) := hofib(F(f):F_{A}(X_\bullet)\to F_{B}(X_\bullet))
\end{gather*}
For all $f:A\to B$, $F_{A}, F_{B}$ and $G$ are reduced homotopy functors on the category of basepointed simplicial sets. The connectivity of $f$ implies the Goodwillie Taylor series of $G$ converges for all $X_\bullet$. The first Goodwillie derivative of $F_{R}$ at an arbitrary space $X_\bullet$, evaluated at $Y_\bullet$, is 
\[
D_1(F_{R})_{X_\bullet}(Y_\bullet) \simeq \left|[q]\mapsto THH(R\wedge T_S(X_q), \left(R\wedge T_S(X_q)\right)[Y_q])\right|
\]
where $R'[S]$ denotes the $R'$-bimodule on the set $S$. For any {\bf S}-algebra $R'$ and $R'$-bimodule $M$, there is an evident equivalence
\[
THH(R',M)\wedge E\simeq THH(R'\wedge E,M\wedge E)
\]
Thus the map on first deriviatives induced by $f$ (at $X_\bullet$, evaluated at $Y_\bullet$)
\[
D_1(f): D_1(F_{A})_{X_\bullet}(Y_\bullet)\to D_1(F_{B})_{X_\bullet}(Y_\bullet)
\]
 becomes an equivalence upon smashing with $E$, implying the $E$-acyclicity of the first derivatives
\begin{equation}
D_1(G)(X_\bullet)(Y_\bullet)\wedge E\simeq *
\end{equation}
for all $X_\bullet, Y_\bullet$. This fact implies that for each $n$ and $Y_\bullet$, the n-th Goodwillie Taylor approximation $P_n(G)(Y_\bullet)$ is $E$-acyclic. In general, homotopy inverse limits of $E$-acyclic spectra need not be themselves $E$-acyclic. However,

\begin{proposition}\label{prop:converge} Suppose $\{Z_n,p_n:Z_n\to Z_{n-1}\}$ is an inverse system of spectra for which
\begin{itemize}
\item The connectivity of $p_n$ tends to infinity with n, and
\item $Z_n$ is $E$-acyclic for each $n$
\end{itemize}
Then $Z = \underleftarrow{holim}\ Z_n$ is $E$-acyclic.
\end{proposition}
\begin{proof} The hypothesis implies the connectivity of the map $Z\to Z_n$ tends to infinity as $n$ increases. The result follows via proof by contradiction for the classical smash-product for spectra (a modification of the original handcrafted smash-product of Boardman) defined in \cite{ja}. This identifies up to natural equivalence with the more recent and smash product construction of \cite{ekmm}, which is what is needed for the case at hand.
\end{proof}

Appealing once more to the convergence of the Goodwillie Tower for $G$, we conclude 
\[
G(X_\bullet)\wedge E\simeq *
\]
for all $X_\bullet$. In particular, this holds for $X = S^0$, completing the proof of the theorem.
\end{proof}

We remark that the connectivity assumption on the homomorphism $f$ is in general necessary for a result like this to hold. An interesting example is provided by the results of \cite{bm1}. Here the map of topological $K$-theory spectra $ku\to KU$ (connective to non-connective) is obviously a $KU$-equivalence. However, the homotopy fiber of the map on $K$-theory identifies with the $K$-theory of $\mathbb Z$, which is rationally non-trivial.
\vskip.2in


Theorem \ref{thm:second} can be applied to yield information on some of the constructions appearing in \cite{fw1}, where Waldhausen explores the relation between (Waldhausen) $K$-theory, the chromatic filtration in stable homotopy, and the Litchenbaum-Quillen conjecture. \vskip.2in

We assume a prime $p$ has been fixed; all results will then be at that prime. Let $L_{p,n}$ denote the n-th Johnson-Wilson localization functor at $p$, and $S_{(p,n)}$ the connective cover of the localized sphere spectrum $E(p,n) = L_{p,n}(S)$. Given a connective {\bf S}-algebra $A$, let $A_{(p,n)} := A\wedge S_{(p,n)}$; this may be thought of as the n-th {\it connective} localization of $A$ at $p$ (given that $L_{p,n}(A) = A\wedge E(p,n)$). In the case $n=0$, we set $A_{(p,0)} := \pi_0(A\wedge H\mathbb Z_{(p)})$. Following \cite{fw1}, we define the {\it integral} or {\it connective chromatic Waldhausen tower for $A$} at the prime $p$ as the inverse system
\[
\dots\to K(A_{(p,n)})\to K(A_{(p,n-1)})\to K(A_{(p,n-2)})\to\dots\to K(A_{(p,0)})
\]
with the {n-th connective monochromatic K-theory of A at $p$} written as
\[
\wt{M}(A,p,n) := hofib(K(A_{(p,n)})\to K(A_{(p,n-1)}))
\]
These spectra are $p$-local (this was already observed in \cite{fw1}; note that Theorem \ref{thm:main} provides a quick proof of this result).

\begin{corollary}\label{cor:first} For $n\ge 1$, $\wt{M}(A,p,n)$ is $E(p,n-1)$-acyclic.
\end{corollary}

\begin{proof} The {\bf S}-algebra map $\phi_{p,n}:A_{(p,n)}\to A_{(p,n-1)}$ is $(2p - 3)$-connected for all $A$, primes $p$, and $n\ge 1$. Moreover, it is an $E(p,n-1)$-equivalence. As $L_{p,m} = L_{E(p,m)}$ is smashing, the result follows by Theorem \ref{thm:second}.
\end{proof}


The same line of reasoning shows that $hofib(K(A_{(p,m)})\to K(A_{(p,n-1)}))$ is $E(p,n-1)$-acyclic for all $n\le m\le \infty$, where $A_{p,\infty} := A_{(p)}$. $K$-theoretic chromatic convergence in this connective setting amounts to asking

\begin{question} Is the natural map $K(A_{(p)})\to \underleftarrow{holim}\ K(A_{(p,n)})$ a weak equivalence of spectra?
\end{question}

We can apply Theorem A to get at least a partial answer to this question. Given a functor $F: {\cal{CSA}}\to (spectra)_*$, we call the tower $\{A_{(p,n)}\}$ {\it F-convergent} if $F(A_{(p)})\to \underleftarrow{holim}\ F(A_{(p,n)})$ is a weak equivalence of spectra.

\begin{proposition} If $\{A_{(p,n)}\}$ is $THH$-convergent, then it is $K$-convergent.
\end{proposition}
\begin{proof} Without loss of generality, we can (by appropriate fibrant replacement) assume the maps in the tower $\{A_{(p,n)}\}$ are fibrant. Applying $THH(_-)$ yields a tower of cyclotomic spectra $\{THH(A_{(p,n)})\}$ and maps $THH(A_{(p,n)})\to THH(A_{(p,n-1)})$ strictly preserving the cyclotomic structure (as defined in \cite{bm2}). Thus the homotopy inverse limit will again admit the structure of a cyclotomic spectrum, whose fixed points resp.\ homotopy fixed points with respect to a subgroup $G$ of $S^1$ are homeomorphic to the homotopy inverse limit of the the inverse systems $\{THH(A_{(p,n)})^G\}$ resp.\ $\{THH(A_{(p,n)})^{hG}\}$. The result is a transformation of the pullback diagram used to define $TC(_-)$ which is a weak equivalence at each vertex, implying the tower is $TC$-convergent. By Theorem A, it is then also $K$-convergent due to the connectivity of the projection maps at each level.
\end{proof}


Although $\wt{M}(A,p,n)$ is $E(p,n-1)$-acyclic, it cannot be $E(p,n)$-local due to its connectivity. However, the trace map induces a map of homotopy fibres
\[
\wt{M}(A,p,n)\overset\simeq{\to} TC(A_{(p,n)}\to A_{(p,n-1)})\to TC(L_{p,n}(A)\to L_{p,n-1}(A))
\]
The topological Hochschild homology spectrum of an $E(p,n)$-local {\bf S}-algebra is $E(p,n)$-local, and locality is preserved under the taking of homotopy inverse limits, as well as homotopy colimits (since $L_{p,n}$ is smashing) \cite{dr}. Thus $TC(L_{p,m}(A))$ is $E(p,n)$-local for all $m\le n$, implying

\begin{proposition} The homotopy fibre $TC(L_{p,n}(A)\to L_{p,n-1}(A))$ is $L_{p,n}$-local for all $A$, primes $p$, and $n\ge 1$.
\end{proposition}

It follows that trace map above, followed by passage to actual localizations of the algebra, induces a transformation of spectra natural in $A$:
\begin{equation}
\lambda(A;p,n):\wt{M}(A,p,n)\wedge E(p,n)\to TC(L_{p,n}(A)\to L_{p,n-1}(A))
\end{equation}

It seems natural to ask

\begin{question} Is $\lambda(A;p,n)$ an equivalence?
\end{question}

If so, Theorem \ref{thm:second} would imply $TC(L_{p,n}(A)\to L_{p,n-1}(A))$ is actually monochromatic of  level $n$, something which does not seem provable directly.
\vskip.2in

\underbar{\bf Remark 2} Again, as in \cite{fw1}, one can define the {\it non-connective} chromatic Waldhausen tower for $A$\footnote{A reasonable construction of the non-connective chromatic tower can be given without the hypothesis that $E(p,n)$ is finite (which it often isn't). A more detailed discussion of this tower and the map from the connective to the non-connective theories will be given elsewhere.}; there is an evident transformation from the connective to the non-connective tower, resulting in a map of monochromatic $K$-theories
\[
\wt{M}(A,p,n) \to M(A,p,n)
\]
As shown in {\it op.~cit.\ }, these theories are quite different, though an exact description of the difference has been elusive (except in certain special cases; compare \cite{bm1}), the main problem being the lack of computational tools for the $K$-theory of non-connective spectra. 
\vskip.3in


\section{The fundamental theorem} If $A$ is an {\bf S}-algebra, then $A[t]$, $A[t^{-1}]$, and $A[t,t^{-1}]$ admit {\bf S}-algebra structures induced by that on $A$ in a natural way. Following \cite[IV.10]{cw2}, define functors from $({\bf S}-algebras)$ to $(spectra)_*$ by
\begin{gather*}
F_0(A) = K(A)\\
F_1(A) = K(A[t])\underset{K(A)}{\vee} K(A[t^{-1}])\\
F_2(A) = K(A[t,t^{-1}])
\end{gather*}
There is an obvious transformation $F_1(_-)\to F_2(_-)$ induced by the inclusions of $A[t]$ and $A[t^{-1}]$ as subalgebras of $A[t,t^{-1}]$, and we set $F_3(_-) := hocofib(F_1(_-)\to F_2(_-))$. For a spectrum $T$, write $\Sigma^{-1} T$ for the desuspension of $T$.

\begin{theorem}\label{thm:fund} For a connective {\bf S}-algebra $A$, there is a map of spectra $K(A)\to \Sigma^{-1} F_3(A)$, functorial in $A$, which induces an equivalence between $K(A)$ and the $(-1)$-connected cover $\Sigma^{-1} F_3(A)<-1>$ of $\Sigma^{-1} F_3(A)$.
\end{theorem}

\begin{proof} Let $\ov{F}_i(A) := F_i(\pi_0(A))$. Projection to $\pi_0$ induces an evident transformation $F_i(_-)\to \ov{F}_i(_-), 0\le i\le 3$; let $\wt{F}_i(A) := hofib(F_i(A)\to \ov{F}_i(A))$.   The canonical element $[t]\in K_1(S[t,t^{-1}])$ represented by the unital element $t$ induces a map
\[
S^1\wedge K(A)\to K(S[t,t^{-1}])\wedge K(A)\to K(A[t,t^{-1}])\to F_3(A)
\]

whose adjoint provides the transformation $F_0(_-) = K(_-)\to \Sigma^{-1} F_3(_-)$.
\vskip.1in
The transformaton $\ov{F}_0(_-)\to \Sigma^{-1}\ov{F}_3(_-)<-1>$ is an equivalence by \cite{dg}, \cite[Thm.\ IV.10.2]{cw2}. It follows that there is a homotopy-Cartesian square
\vskip.2in
\centerline{
\xymatrix{
\wt{F}_0(_-)\ar[r]\ar[d] & \Sigma^{-1}\wt{F}_3(_-)<-1>\ar[d]\\
F_0(_-)\ar[r] & \Sigma^{-1} F_3(_-)<-1>
}}
\vskip.2in

Consequently, if the top horizontal map is an equivalence, then so is the lower one. However, for the functors on the top row, we are within the \lq\lq radius of convergence\rq\rq\ of Goodwillie Calculus. Precisely, both functors are homotopy functors, and preserve connectivity (up to a shift by some universal constant), and so the same three claims used in the proof of Theorem \ref{thm:main} apply in this setting as well, reducing the proof to verification that the transformation
\[
\wt{F}_0(A_{\bullet}\ltimes A_{\bullet}[Y_{\bullet}]\surj A_{\bullet})\to \Sigma^{-1}\wt{F}_3(A_{\bullet}\ltimes A_{\bullet}[Y_{\bullet}]\surj A_{\bullet})<-1>
\]
is a weak equivalence for connected $Y_{\bullet}$. Note also that the connectivity of $Y_{\bullet}$ here implies
\[
\Sigma^{-1}\wt{F}_3(A_{\bullet}\ltimes A_{\bullet}[Y_{\bullet}]\surj A_{\bullet})<-1> = \Sigma^{-1}\wt{F}_3(A_{\bullet}\ltimes A_{\bullet}[Y_{\bullet}]\surj A_{\bullet})
\]
Appealing once more to the methods and results of \cite{lm}, we see that both functors are 0-analytic in the $Y_{\bullet}$-coordinate, reducing the proof to a comparison of Goodwillie derivatives. Taking derivatives commutes with homotopy colimits, so the computation of \cite{lm} applies, from which we conclude that the map on coefficient spectra of n-th derivatives is given by
\begin{multline}
sd_n THH(A_{\bullet};A_{\bullet})\to\\
\Sigma^{-1}hocofib\bigg((sd_n THH(A_{\bullet}[t];A_{\bullet}[t])\underset{sd_n THH(A_{\bullet};A_{\bullet})}\vee
sd_n THH(A_{\bullet}[t^{-1}];A_{\bullet}[t^{-1}])\\
\to
sd_n THH(A_{\bullet}[t,t^{-1}];A_{\bullet}[t,t^{-1}])\bigg)
\end{multline}
For convenience, write the left-hand side as $G_1(A_{\bullet})$, and the right-hand side as $G_2(A_{\bullet})$, where $G_i(_-)$ may be applied to general {\bf S}-algebras. Then there is a canonical equivalence $G_i(A_{\bullet})\simeq THH(A_{\bullet})\wedge G_i(S)$, which further reduces the problem to verifying the map is an equivalence for the sphere spectrum $S$. But $G_2(S)$ is simply the suspension spectrum of the free loop space of $S^1$: $\Sigma^{\infty}((S^1)^{S^1})$. Here one has the stable splitting
\[
\Sigma^{\infty}((S^1)^{S^1})\simeq \Sigma^{\infty}(S^1)\vee \Sigma^{\infty}((S^1)^{S^1}/S^1)
\]
with the factor $\Sigma^{\infty}(S^1)\simeq |\Sigma sd_n THH(S;S)|$ coresponding to the suspension of $G_2(S)$, and the map $G_1(S)\to G_2(S)$ corresponding to the equivalence
\[
\Sigma^{\infty}(S^0)\overset{\simeq}{\longrightarrow} \Sigma^{-1}\Sigma^{\infty}(S_1)
\]
\end{proof}

\underbar{\bf Remark 3} In \cite[\S 9]{bm3}, Blumberg and Mandell coin the term {\it Bass functor} for homotopy functors exhibiting the above type of behavior. In particular, they show that the topological Dennis trace $K(-)\to THH(-)$ is a transformation of Bass functors, at least for discrete rings. The above suggests that this particular result of theirs extends to the category of {\bf S}-algebras.
\vskip.2in

A consequence of this theorem is that the usual machinery associated with a spectral interpretation of the Fundamental Theorem produces a natural {\it non-connective} delooping of the $K$-theory functor $A\mapsto K(A)$ on the category $\cal{CSA}$, via iterated application of the natural transformation $K(_-)\to \Sigma^{-1}F_3(_-)$. The result is a (potentially) non-connective functor
\[
A\mapsto K^B(A)
\]
differing from the deloopings arising from the \lq\lq plus\rq\rq\ construction \cite{ekmm}, or iterations of Waldhausen's $wS.$-construction \cite{fw2}, which are always connective. Notice, however, that the connectivity of the map $A\to \pi_0(A)$ implies that (at least for connective {\bf S}-algebras) the negative $K$-groups arising from iteration of the above construction \cite[Cor.\ IV.10.3]{cw2} depend only on $\pi_0(A)$. Combining with the equivalence one has for low-dimensional $K$-groups, this can be summarized by the equality
\[
K_{n}(A) = K_{n}(\pi_0(A)),\,\, n\le 1
\]

\begin{definition} The NK-spectrum of an {\bf S}-algebra $A$ is $NK(A) := hofib(K^B(A[t])\to K^B(A))$.
\end{definition}

To make the notation correspond with convention, we should set $NK^+(A) := NK(A)$ as just defined, and $NK^-(A) := hofib(K^B(A[t^{-1}])\to K^B(A))$. In this way, we arrive at a more conventional formulation of Theorem \ref{thm:fund}:

\begin{theorem} For a connective {\bf S}-algebra $A$, there is a functorial splitting of spectra
\[
K^B(A[t,t^{-1}])\simeq K^B(A)\vee \Omega^{-1}(K^B(A))\vee NK^+(A)\vee NK^-(A)
\]
where $\Omega^{-1}(K^B(A))$ denotes the non-connective delooping of $K^B(A)$ indicated above. Moreover, the involution $t\mapsto t^{-1}$ induces an involution on $K^B(A[t,t^{-1}])$ which acts as the identity on the first two factors and switches the second two factors.
\end{theorem}

In the particular case $A = \Sigma^{\infty}(\Omega(X)_+)$ for a connected basepointed space $X$, we recover the main results of \cite{hkvww1, hkvww2}.
\vskip.2in

Given the difficulty of computing $NK_*(R)$ for discrete rings, it is not surprising that not much is known about $NK(A)$ for general {\bf S}-algebras $A$. In the discrete setting, it is a classical result of Quillen that $R$ Noetherian regular implies $NK(R)\simeq *$. This fact led to the notion of {\it $NK$-regularity}; rings whose $NK$-spectrum was contractible. Via the above discussion, the same notion of $NK$-regularity may be extended to arbitrary {\bf S}-algebras.
\vskip.1in

It has been shown by Klein and Williams \cite{kw1} that the map of Waldhausen spaces arising from the Fundamental Theorem of \cite{hkvww1} (and temporarily writing $A(X)$ for the Waldhausen $K$-theory of the space $X$)
\[
A(*)\vee \Omega^{-1} A(*)\to A(S^1)
\]
is the inclusion of a summand but not an equivalence. In the notation used here, $A(*) = K({\bf S})$ and $A(S^1) = K({\bf S}[t,t^{-1}])$, where ${\bf S}$ denotes the sphere spectrum. Thus (unlike the case of the discrete ring $\mathbb Z$), one has

\begin{corollary} The sphere spectrum {\bf S} is not $NK$-regular.
\end{corollary}

\end{document}